\documentclass[12pt]{article}
\usepackage[english]{babel}
\usepackage{amssymb,amsmath,amsthm}
\usepackage{graphicx}
\usepackage{graphicx}
\usepackage{amsmath}
\usepackage{amsthm}
\usepackage{amsfonts}
\usepackage[dvipsnames]{pstricks}
\usepackage{amssymb, amscd}

\usepackage{euscript}
\usepackage[all,tips]{xy}

\usepackage{color}

\newtheorem{theorem}{Theorem}[section]
\newtheorem{lemma}[theorem]{Lemma}
\newtheorem{proposition}[theorem]{Proposition}
\newtheorem{corollary}[theorem]{Corollary}

\newtheorem{example}[theorem]{Example}

{\theoremstyle{definition}}
{\theoremstyle{definition}\newtheorem{definition}[theorem]{Definition}}
{\theoremstyle{definition}}

\numberwithin{equation}{section}

\def\epsilon{\varepsilon}
\def\kappa{\varkappa}
\def\phi{\varphi}
\def\leq{\leqslant}
\def\geq{\geqslant}

\def\dim{{\rm dim}\,}

\def\ker{\hbox{\tt ker}\,}

\def\im{\hbox{\tt im}\,}

\def\d{\hbox{\tt dim}\,}

\def\lll{\langle}
\def\rrr{\rangle}

\def\o{\otimes}
\def\ep{\epsilon}

\def\PICONE{{\hbox{
\begin{pspicture}(0,0)(7,2.1)
\psline[linewidth=1.2pt](0.2,1.4)(0.5,2.1)(1.4,0)
\psline[linewidth=1.2pt](0.8,1.4)(0.5,0.7)
\psline[linewidth=1.2pt](1.1,0.7)(0.8,0)
\put(1.5,0){,}
\psline[linewidth=1.2pt](1.7,0)(2.3,1.4)(2.9,0)
\psline[linewidth=1.2pt](2,0.7)(2.2,0)
\psline[linewidth=1.2pt](2.6,0.7)(2.4,0)
\put(3,0){,}
\psline[linewidth=1.2pt](3.2,0)(4.1,2.1)(4.4,1.4)
\psline[linewidth=1.2pt](3.8,1.4)(4.1,0.7)
\psline[linewidth=1.2pt](3.5,0.7)(3.8,0)
\put(4.5,0){,}
\psline[linewidth=1.2pt](5.6,0)(5,1.4)(5.3,2.1)(5.6,1.4)
\psline[linewidth=1.2pt](5.3,0.7)(5,0)
\psline[linewidth=1.2pt](4.7,0.7)(5,1.4)
\put(5.7,0){,}
\psline[linewidth=1.2pt](6,0)(6.6,1.4)(6.3,2.1)(6,1.4)
\psline[linewidth=1.2pt](6.6,1.4)(6.9,0.7)
\psline[linewidth=1.2pt](6.3,0.7)(6.6,0)
\end{pspicture}
}}}

\def\PICTWO{{\hbox{
\begin{pspicture}(0,0)(2.2,2.1)
\psline[linewidth=1.8pt](0,0)(0.2,0.7)(0.4,0)
\psline[linewidth=1.8pt](1.8,0)(2,0.7)(2.2,0)
\psline[linewidth=1.8pt](0.6,0)(0.8,0.7)(1,0)
\psline[linewidth=1.8pt](1.6,0)(1.4,0.7)(1.2,0)
\psline[linewidth=1.2pt](0,0)(0.2,0.7)(0.5,1.4)(1.1,2.1)(1.7,1.4)(2,0.7)(2.2,0)
\psline[linewidth=1.2pt](0.2,0.7)(0.4,0)
\psline[linewidth=1.2pt](2,0.7)(1.8,0)
\psline[linewidth=1.2pt](0.5,1.4)(0.8,0.7)(1,0)
\psline[linewidth=1.2pt](0.8,0.7)(0.6,0)
\psline[linewidth=1.2pt](1.7,1.4)(1.4,0.7)(1.2,0)
\psline[linewidth=1.2pt](1.4,0.7)(1.6,0)
\end{pspicture}
}}}

\def\PICTHREE{{\hbox{
\begin{pspicture}(0,-.5)(7.5,3)
\psline[linewidth=1.2pt](0,0)(0.1,0.7)(0.3,1.4)(0.7,2.1)(1.5,2.8)(2.3,2.1)(2.7,1.4)(2.9,0.7)(3,0)
\psline[linewidth=1.2pt](0.1,0.7)(0.2,0)
\psline[linewidth=1.2pt](2.9,0.7)(2.8,0)
\psline[linewidth=1.2pt](0.3,1.4)(0.5,0.7)(0.6,0)
\psline[linewidth=1.2pt](0.5,0.7)(0.4,0)
\psline[linewidth=1.2pt](0.7,2.1)(1.1,1.4)(1.3,0.7)(1.4,0)
\psline[linewidth=1.2pt](1.1,1.4)(0.9,0.7)(0.8,0)
\psline[linewidth=1.2pt](0.9,0.7)(1,0)
\psline[linewidth=1.2pt](2.3,2.1)(1.9,1.4)(1.7,0.7)(1.6,0)
\psline[linewidth=1.2pt](1.7,0.7)(1.8,0)
\psline[linewidth=1.2pt](1.3,0.7)(1.2,0)
\psline[linewidth=1.2pt](1.9,1.4)(2.1,0.7)(2.2,0)
\psline[linewidth=1.2pt](2.1,0.7)(2,0)
\psline[linewidth=1.2pt](2.7,1.4)(2.5,0.7)(2.4,0)
\psline[linewidth=1.2pt](2.5,0.7)(2.6,0)
\put(1.2,-0.3){{\Large.\ .\ .}}
\put(0.3,2.5){$T_n$}
\psline[linewidth=1.2pt](4.3,1.5)(4.6,0.5)(4,0.5)(4.3,1.5)(5.1,1.5)(4.8,0.5)(5.4,0.5)(5.1,1.5)(4.7,2.5)(4.3,1.5)
\psline[linewidth=1.2pt](4.7,2.5)(6.4,2.5)(6,1.5)(5.7,0.5)(6.3,0.5)(6,1.5)(6.8,1.5)(6.5,0.5)(7.1,0.5)(6.8,1.5)(6.4,2.5)
\put(7.2,2){${\rm GR}(T_n)$}
\put(5.2,0.2){{\Large.\ .\ .}}
\pscircle[linewidth=2pt](4,0.5){0.07}
\pscircle[linewidth=2pt](4.6,0.5){0.07}
\pscircle[linewidth=2pt](4.3,1.5){0.07}
\pscircle[linewidth=2pt](4.8,0.5){0.07}
\pscircle[linewidth=2pt](5.4,0.5){0.07}
\pscircle[linewidth=2pt](5.1,1.5){0.07}
\pscircle[linewidth=2pt](5.7,0.5){0.07}
\pscircle[linewidth=2pt](6.3,0.5){0.07}
\pscircle[linewidth=2pt](6,1.5){0.07}
\pscircle[linewidth=2pt](6.5,0.5){0.07}
\pscircle[linewidth=2pt](7.1,0.5){0.07}
\pscircle[linewidth=2pt](6.8,1.5){0.07}
\pscircle[linewidth=2pt](4.7,2.5){0.07}
\pscircle[linewidth=2pt](6.4,2.5){0.07}
\end{pspicture}
}}}

\def\PICFOUR{{\hbox{
\begin{pspicture}(0,0)(2.2,2.1)
\psline[linewidth=1.2pt](0,0)(0.2,0.7)(0.5,1.4)(1.1,2.1)(1.7,1.4)(2,0.7)(2.2,0)
\psline[linewidth=1.2pt](0.2,0.7)(0.4,0)
\psline[linewidth=1.2pt](2,0.7)(1.8,0)
\psline[linewidth=1.2pt](0.5,1.4)(0.8,0.7)(1,0)
\psline[linewidth=1.2pt](0.8,0.7)(0.6,0)
\psline[linewidth=1.2pt](1.7,1.4)(1.4,0.7)(1.2,0)
\psline[linewidth=1.2pt](1.4,0.7)(1.6,0)
\put(0.3,1){$\!\!\!\!\!\!\!\!\!T_3$}
\psline[linewidth=1.2pt](3.7,1.5)(3.3,0.5)(4.1,0.5)(3.7,1.5)(5.4,1.5)(5,0.5)(5.8,0.5)(5.4,1.5)
\put(5.8,1.4){${\rm GR}(T_3)$}
\pscircle[linewidth=2pt](3.3,0.5){0.07}
\pscircle[linewidth=2pt](4.1,0.5){0.07}
\pscircle[linewidth=2pt](5,0.5){0.07}
\pscircle[linewidth=2pt](5.8,0.5){0.07}
\pscircle[linewidth=2pt](3.7,1.5){0.07}
\pscircle[linewidth=2pt](5.4,1.5){0.07}
\end{pspicture}
}}}

\def\PICFIVE{{\hbox{
\begin{pspicture}(0,-0.5)(11,2)
\put(0,0.1){\Huge H(}
\pscircle[linewidth=2pt](1,0){0.06}
\pscircle[linewidth=2pt](1.8,0){0.06}
\pscircle[linewidth=2pt](2.4,0){0.06}
\pscircle[linewidth=2pt](3.2,0){0.06}
\pscircle[linewidth=2pt](1.4,0.8){0.06}
\pscircle[linewidth=2pt](2.8,0.8){0.06}
\psline[linewidth=1.2pt](1.4,0.8)(1,0)(1.8,0)(1.4,0.8)(2.8,0.8)(2.4,0)(3.2,0)(2.8,0.8)
\put(3.3,0.1){\Huge )\ =H(}
\put(0.9,-0.3){$a$}
\put(1.7,-0.3){$b$}
\put(1.3,0.9){$c$}
\pscircle[linewidth=2pt](5.5,0){0.06}
\pscircle[linewidth=2pt](6.3,0){0.06}
\pscircle[linewidth=2pt](5.9,0.8){0.06}
\psline[linewidth=1.2pt](5.5,0)(6.3,0)(5.9,0.8)(5.5,0)
\put(6.4,0.1){\Huge )\ +\ H(}
\pscircle[linewidth=2pt](8.9,0){0.06}
\pscircle[linewidth=2pt](9.6,0){0.06}
\psline[linewidth=1.2pt](8.9,0)(9.6,0)
\put(9.7,0.1){{\Huge )}\Large\ =\ 2+1\ =\ 3}
\put(7.5,-0.45){$c$ eliminated}
\put(4.4,-0.45){$b$ eliminated}
\end{pspicture}
}}}

\def\PICTEN{{\hbox{
\begin{pspicture}(0,-0.5)(7,2)
\put(0,0.1){\Huge w\ =}
\psline[linewidth=1.2pt](2,0)(2.3,0.6)(2.6,0)
\psline[linewidth=1.2pt](3,0)(3.3,0.6)(3.6,0)
\psline[linewidth=1.2pt](4,0)(4.3,0.6)(4.6,0)
\psline[linewidth=1.2pt](6,0)(6.3,0.6)(6.6,0)
\put(4.7,0.5){\Huge .\ .\ .}
\put(6.5,0.5){$n$}
\put(1.9,0.5){$1$}
\put(2.9,0.5){$2$}
\put(3.9,0.5){$3$}
\psline[linewidth=1.2pt](2.3,0.6)(3.5,1.7)
\psline[linewidth=1.2pt](3.3,0.6)(3.5,1.7)
\psline[linewidth=1.2pt](4.3,0.6)(3.5,1.7)
\psline[linewidth=1.2pt](6.3,0.6)(3.5,1.7)
\end{pspicture}
}}}

\def\PICELEVEN{{\hbox{
\begin{pspicture}(0,-0.5)(9,2)
\put(0,0.1){\Huge w\ =}
\psline[linewidth=1.2pt](2,0)(2.3,0.6)(2.6,0)
\psline[linewidth=2.1pt](3,0)(3.3,0.6)(3.6,0)
\psline[linewidth=2.1pt](4,0)(4.3,0.6)(4.6,0)
\psline[linewidth=2.1pt](6,0)(6.3,0.6)(6.6,0)
\psline[linewidth=1.2pt](8,0)(8.3,0.6)(8.6,0)
\put(4.7,0.5){\Huge .\ .\ .}
\put(6.7,0.5){\Huge .\ .\ .}
\put(6.42,0.5){$i_k$}
\put(8.5,0.5){$n$}
\put(1.9,0.5){$1$}
\put(2.9,0.5){$i_1$}
\put(3.9,0.5){$i_2$}
\psline[linewidth=1.2pt](2.3,0.6)(3.5,1.7)
\psline[linewidth=2.1pt](3.3,0.6)(3.5,1.7)
\psline[linewidth=2.1pt](4.3,0.6)(3.5,1.7)
\psline[linewidth=2.1pt](6.3,0.6)(3.5,1.7)
\psline[linewidth=1.2pt](8.3,0.6)(3.5,1.7)
\end{pspicture}
}}}

\def\PICTWELVE{{\hbox{
\begin{pspicture}(0,-0.5)(7,2)
\put(0,0.1){\Huge w\ =}
\psline[linewidth=1.2pt](2,0)(2.3,0.6)(2.6,0)
\psline[linewidth=1.2pt](3,0)(3.3,0.6)(3.6,0)
\psline[linewidth=1.2pt](4,0)(4.3,0.6)(4.6,0)
\psline[linewidth=1.2pt](5,0)(5.3,0.6)(5.6,0)
\put(4.9,0.45){$4$}
\put(1.9,0.5){$1$}
\put(2.9,0.5){$2$}
\put(3.9,0.5){$3$}
\psline[linewidth=1.2pt](2.3,0.6)(3.5,1.7)
\psline[linewidth=1.2pt](3.3,0.6)(3.5,1.7)
\psline[linewidth=1.2pt](4.3,0.6)(3.5,1.7)
\psline[linewidth=1.2pt](5.3,0.6)(3.5,1.7)
\end{pspicture}
}}}

\def\PICTHIRTEEN{{\hbox{
\begin{pspicture}(0,0)(8,2.1)
\psline[linewidth=1.2pt](0.2,1.4)(0.5,2.1)(1.4,0)
\psline[linewidth=1.2pt](0.8,1.4)(0.5,0.7)
\psline[linewidth=1.2pt](1.1,0.7)(0.8,0)
\psline[linewidth=1.2pt](2,0)(2.6,1.4)(2.3,2.1)(2,1.4)
\psline[linewidth=1.2pt](2.6,1.4)(2.9,0.7)
\psline[linewidth=1.2pt](2.3,0.7)(2.6,0)
\psline[linewidth=1.2pt](3.5,0)(4.4,2.1)(4.7,1.4)
\psline[linewidth=1.2pt](4.1,1.4)(4.4,0.7)
\psline[linewidth=1.2pt](3.8,0.7)(4.1,0)
\psline[linewidth=1.2pt](6.2,0)(5.6,1.4)(5.9,2.1)(6.2,1.4)
\psline[linewidth=1.2pt](5.9,0.7)(5.6,0)
\psline[linewidth=1.2pt](5.3,0.7)(5.6,1.4)
\psline[linewidth=1.2pt](6.7,0)(7.3,1.4)(7.9,0)
\psline[linewidth=1.2pt](7,0.7)(7.2,0)
\psline[linewidth=1.2pt](7.6,0.7)(7.4,0)
\end{pspicture}
}}}

\def\PICFOURTEEN{{\hbox{
\begin{pspicture}(0,0)(2.2,2.1)
\psline[linewidth=1.2pt](0,0)(0.2,0.7)(0.5,1.4)(1.1,2.1)(1.7,1.4)(2,0.7)(2.2,0)
\psline[linewidth=1.2pt](0.2,0.7)(0.4,0)
\psline[linewidth=1.2pt](2,0.7)(1.8,0)
\psline[linewidth=1.2pt](0.5,1.4)(0.8,0.7)(1,0)
\psline[linewidth=1.2pt](0.8,0.7)(0.6,0)
\psline[linewidth=1.2pt](1.7,1.4)(1.4,0.7)(1.2,0)
\psline[linewidth=1.2pt](1.4,0.7)(1.6,0)
\put(0.25,1.3){$1$}
\put(1.8,1.3){$2$}
\put(-0.07,0.6){$3$}
\put(0.53,0.6){$4$}
\put(1.5,0.6){$5$}
\put(2.07,0.6){$6$}
\end{pspicture}
}}}

\def\PICFIFTEEN{{\hbox{
\begin{pspicture}(-.5,0)(14,8)
\psline[linewidth=1.2pt](2.5,3)(2,2)(3,2)(2.5,3)(4.5,3)(4,2)(5,2)(4.5,3)
\psline[linewidth=1.2pt](6.5,3)(6,2)(7,2)(6.5,3)(7.7,3)
\psline[linewidth=1.2pt](13.5,3)(14,2)(13,2)(13.5,3)(11.5,3)(12,2)(11,2)(11.5,3)
\psline[linewidth=0.9pt](3,2)(3,1.5)
\psline[linewidth=0.9pt](4,2)(4,1)
\psline[linewidth=0.9pt](6,2)(6,1)
\psline[linewidth=0.9pt](7,2)(7,1.5)
\psline[linewidth=0.9pt](11,2)(11,1.5)
\psline[linewidth=0.9pt](13,2)(13,1)
\pscircle[linewidth=1.5pt](4,1.5){0.06}
\pscircle[linewidth=1.5pt](6,1.5){0.06}
\pscircle[linewidth=1.5pt](13,1.5){0.06}
\put(0,0.7){$\!\!w(x,y,z)=$}
\put(1.9,0.7){$x$}
\put(2.9,0.7){$y$}
\put(3.9,0.7){$z$}
\put(4.9,0.7){$x$}
\put(5.9,0.7){$z$}
\put(6.9,0.7){$y$}
\put(13.9,0.7){$x$}
\put(12.9,0.7){$z$}
\put(11.9,0.7){$x$}
\put(10.9,0.7){$y$}
\pscircle[linewidth=1.5pt](8.3,0.7){0.05}
\pscircle[linewidth=1.5pt](9.25,0.7){0.05}
\pscircle[linewidth=1.5pt](10.2,0.7){0.05}
\psline[linewidth=1.2pt](4,6)(3,4.5)(5,4.5)(4,6)(6.7,6)(5.7,4.5)(7.7,4.5)(6.7,6)(5.35,8)(4,6)
\psline[linewidth=1.2pt](9.3,6)(8.3,4.5)(10.3,4.5)(9.3,6)(12,6)(11,4.5)(13,4.5)(12,6)(10.65,8)(9.3,6)
\psline[linewidth=1.2pt](5.35,8)(10.65,8)
\psline[linewidth=1pt,linestyle=dotted](2.5,3)(0.5,3)
\psline[linewidth=1pt,linestyle=dotted](3,4.5)(0.5,4.5)
\psline[linewidth=1pt,linestyle=dotted](4,6)(0.5,6)
\psline[linewidth=1pt,linestyle=dotted](5.35,8)(0.5,8)
\put(0.2,2.9){$n$}
\put(0.2,4.4){$2$}
\put(0.2,5.9){$1$}
\put(0.2,7.9){$0$}
\pscircle[linewidth=1.5pt](5,3.8){0.07}
\pscircle[linewidth=1.5pt](8,3.8){0.07}
\pscircle[linewidth=1.5pt](11,3.8){0.07}
\end{pspicture}
}}}

\title{Homologies of monomial operads and algebras.
}

\author{Natalia Iyudu, Ioannis Vlassopoulos }

\date{}

\begin{document}

\maketitle

\bigskip
\centerline{\it Dedicated to the memory of Victor Nikolaevich Latyshev}
\bigskip

\begin{abstract} We consider the bar complex of a monomial non-unital associative algebra $A=k \lll X \rrr / (w_1,...,w_t)$. It splits as a  direct sum of complexes $B_w$, defined for any fixed monomial $w=x_1...x_n \in A$.
 We give a simple argument, showing that the homology of this subcomplex is at most one-dimensional, and describe the place where the nontrivial homology appears. It has a  very simple expression in terms of the length of the generalized Dyck path associated to a given monomial in $w \in A$.

  The operadic analogue  of the question about  dichotomy in homology is considered. It is shown that dichotomy holds in case when monomial tree-relations form an order. Examples are given showing that in general dichotomy and homological purity does not hold. For quadratic operads, the combinatorial tool for calculating homology in terms of relation graphs is developed. Example of using these methods to compute homology in truncated binary operads is given.
\end{abstract}



\normalsize

\section{Introduction}

Let $A$  be a monomial algebra without a unit. We fix a presentation of $A$ by generators and monomial relations $w_1,...,w_t: \,\,\, A=k \lll X \rrr / (w_1,...,w_t)$, and suppose that monomials $w_1,...,w_t$ have the property that no monomial is a submonomial of another one. Moreover, suppose naturally  that this set of monomials does not contain any generator from $X$. An introduction to monomial algebras and their structural and homological properties can be found in \cite{L}. Note that the algebras we consider here are non-unital, therefore,  the bar complex can have a nontrivial homology.

 Consider the following subcomplex $B_w$ of the bar complex, associated to a fixed monomial $w=x_1...x_n$.
 $$0 \longrightarrow B_n=\{  x_1\o...\o x_n \}  \mathop{\longrightarrow}\limits^D B_{n-1} =\{  x_1\o...\o x_i x_{i+1} \o ... \o x_n \} \mathop{\longrightarrow}\limits^D ...   \mathop{\longrightarrow}\limits^D B_1 =\{  x_1... x_n \}   \longrightarrow 0
 $$

 Clearly, $B_i \subset A^{\o i}$ and the direct sum of these subcomplexes  of the bar complex, for all words $w$ in the monomial algebra $A$  gives the bar complex. Consider
 $B=\oplus B_i$ as a graded linear space. We have a linear map $D: B \longrightarrow B$, satisfying $ D^2=0. $ This map sends any monomial $ u_1\o...\o u_k$ from $B_k$ to the linear combination of monomials with one deleted tensor:
$$D( u_1\o...\o u_k)= \sum_i(-1)^{\sigma} u_1\o...\o u_i u_{i+1} \o ... \o u_k$$
 First, in section~\ref{sec2} we will give a simple proof for the statement
 that homologies of the defined above subcomplex of bar complex are at most one-dimensional.
 In section~\ref{sec5} this proof is generalised to a certain class of operads, and it is shown that the 0-1 dichotomy for homology does not hold in general for monomial operads.

 The next step is to find the place in the complex, where nonzero homology appears,
 and express the result in combinatorial terms related  to the monomial algebra data. Namely we use just the {\it length} of a generalised {\it Dyck path} defined by a word in a  monomial algebra. It is done in section~3.

 Then we consider analogous questions for operads.
 First, we show that for operads given by order, the  0-1 dichotomy in homology still holds.
 However, we present examples, showing that in the operadic case not only does the 0-1 dichotomy in homology  not hold, but also that a conjecture \cite{D} that this homology is pure will fail.

 In  section~\ref{sec7} we clarify and simplify the picture of operadic homology, and to give a convenient tool for the calculation of operadic homology for monomial operads, we first explain correspondence between the latter and the homologies of certain complexes for monomial factors of Grassmann algebras.

Then, in section~\ref{sec9} we consider the case of quadratic operads. For them we define a graph encoding the relations. In terms of these graphs we describe certain transformations (Theorem~\ref{9.3}) that allow us
  to 'split' the complex into 'smaller' ones, and by these means calculate homology.
 Using this tool we can, for example, recursively  calculate homologies for the operad of truncated binary trees.
 The recursion involves parametrisation by noncommutative monomials, which is (mild for this example) manifestation of a general parametrisation by trees.

 The version of this text is publushed in the  MPIM preprint series \cite{Prepr}.

\section{Homologies of the subcomplex of the bar complex, defined by a monomial, are at most one-dimensional}\label{sec2}

The main goal of this section is to give a simple direct proof of the following statement.

 \begin{theorem}\label{th1.1}
 The full homology of the defined above subcomplex $B_w$ of the bar complex of the non-unital monomial algebra $A$ is at most one-dimensional:
 $$ \dim H_{\bullet}(B_w)={\dim} {\ker} D / {\im} D \in \{0,1\}$$
 \end{theorem}

 \begin{proof}

 (induction by n)

 Basis: if the word $w$ is empty, the complex is zero and $H_w=0$. If $w=x$ is a letter, the complex is $0 \longrightarrow k \longrightarrow 0$, and $H_w=k$ is one-dimensional. If $w=xy$, then in case $xy \neq 0$ in $A$, the complex is $0 \longrightarrow (x \o y) k \longrightarrow xy k \longrightarrow 0$.
 It is exact, $H_w=0$. In case $xy=0$, the complex is
 $0 \longrightarrow (x\o y) k \longrightarrow 0$, and $H_w=k$.

 We need to prove that
 ${\dim} \, {\ker} D / {\im} D \in \{0,1\}.$

 Since
 $ {\dim}  {\im} D   = {\dim} B -     {\dim} {\ker} D,
 $
 it is the same as
 $$ {\dim} {\ker} D -  {\dim} {\im} D = $$
 $$ 2 {\dim} {\ker} D -  {\dim} B \in \{0,1\}
 $$

 Thus, $\dim H_{\bullet}(B_w,D)=2 {\dim} {\ker} D -  {\dim} B $.

 Let us split the vector space $B$ as $B=E \oplus F$, where $E=\{x_1\o...\}_k$ and
 $F=\{x_1...\}_k$ are subspaces spanned by those monomials where the first letter is followed directly by a tensor symbol, and by those where the tensor either appears later or  not  at all.
 Consider the linear map
 $ J: E \longrightarrow F$, defined by $x_1\o u \mapsto x_1u$.

 The kernel of this map is a linear span of monomials of the type $x_1\o x_2...x_s\star ...$,
 where $
 x_1x_2...x_s $ is the shortest beginning subword of $w$, which is zero in $A$. That is
 $x_1x_2...x_s=0$, but $x_1x_2...x_{s-1} \neq 0$, and $\star$ throughout this paper means that in this place and further in the word operations are either $\o$ or just multiplication $\cdot$ in the monomial algebra $A$:
 $${\ker} J=\{x_1\o x_2...x_s\star... \}_k$$
 Denote by $L=\{x_2...x_s\star... \}_k$, where $x_2...x_s$ as above, so ${\ker} J=x_1 \o L$.

 First we calculate the images of differential $D$ on elements from $E$ and $F$, namely:
 $$D(x_1\o u)= x_1u- x_1 \o du= J(x_1 \o u) - x_1\o du$$
 Here $d$ is the same differential as $D$, but acts on the word $x_2...x_n$.
$$D(x_1 u)= x_1 du =J (x_1 \o du)$$

Now calculate  the ${\ker} D$. Let $\ep \in B$, such that $D(\ep)= 0$.
 We have $B=E \oplus F$, so $\ep $ is uniquely presented as:
$$ \ep = x_1 u + x_1 \o v$$
 Applying $D$ to elements from $E$ and $F$ as above, we have
 $$0= D(\ep)= J(x_1\o du)+J(x_1 \o v) - x_1 \o dv$$
 Note that here $J(x_1\o du)+J(x_1 \o v) \in F$ and $x_1 \o dv \in E$.
 Hence this is equivalent to $x_1 \o dv = 0 $ and $J(x_1 \o (du + v))=0$,
 which means $dv=0$ and $ x_1 \o (du + v) \in {\ker} J=x_1 \o L$.
 Again, the latter two conditions are equivalent to $dv=0$ and $du+v \in L$, which in turn can be reformulated as $dv=0$ and $v\in -du+ L$. The latter two mean $v\in -du+ (L \cap {\ker} d)$.

 Now
 $$ {\rm dim} \,{\ker} D = \d F + \d(L \cap \ker d)= \d E-\d L+\d(L \cap \ker d).$$
 Thus
 $$ {\rm dim} H_{\bullet}(B_w,D)=2\d \ker D -\d B= $$
 $$2 \d E- 2\d L + 2\d (L \cap \ker d)- \d E- \d E +\d L= $$
 $$2\d (L \cap \ker d) - \d L = {\rm dim} H_{\bullet}(L_{w'},d)$$
 Here $d$ is a typical bar differential on the subword of $w$ - starting from $x_2$, $w'=ux_{s+1}...x_n$ - and $L_{w'}$ is a bar subcomplex, defined by the word $w'$.

 We got that $H_{\bullet}(B_w,D) =  H_{\bullet}(L_{w'},d)$, the latter by the inductive assumption is in $\{0,1\}$, thus so is $H_{\bullet}(B_w,D)$.
\end{proof}

Let us formulate here a corollary, which will be used in section \ref{sec3}.

\begin{corollary}
For any word $w$ of noncommutative non-unital algebra $A$, the defined above subcomplex $B_w$ of the bar complex is exact if and only if the dimension of $B_w$ is even.
\end{corollary}

\begin{proof}
For any complex $B$ it is true, that its dimension is even or odd together with the dimension of its homology. Indeed,
$$H_{\bullet}(B)=\ker D/ \im D, \,\,\, B=\ker D + \im D,$$
so,
$$\dim H_{\bullet}(B)=\dim \ker D - \dim \im D, \,\,\, \dim B=\dim \ker D + \dim \im D,$$
thus, $\dim H_{\bullet}(B)-\dim B = -2 \dim \im D$ is even.

Taking into account Theorem~\ref{th1.1}, we see that since $ \dim H_{\bullet}(B_w)$ can be either $0$ or $1$, the subcomplex $B_w$ is exact if and only if the $\dim B $
is even.
\end{proof}

{\bf Remark} It worth comparing our arguments with those of discrete Morse theory \cite{M}, where the combinatorial condition on  the maps between the linear basis of the complex is formulated, which allows to find a subcomplex with the same homologies. It is easy to see that for the obvious linear basis in the case of our subcomplex, the condition is not satisfied, so the Morse theory can not be applied, at least not in an obvious way.

 \section{The exact position of nontrivial homology in the bar subcomplex of a monomial algebra}\label{sec3}

 As above let $B_w$ be the subcomplex of the bar complex of monomial non-unital algebra $A$,
 associated with the monomial $w=x_1...x_n \in \lll X \rrr.$

In this section for any word $w =x_1...x_n$ and a fixed set of generating monomials, also known as relations, $w_1,...,w_t$ we will define a {\it generalised Dyck path } as follows. Take the first (minimal) beginning subword
 $x_1...x_{d_1}$ of $x_1...x_n$ that is zero in $A$, i.e.  contains as a submonomial one of the monomials
 $w_1,...,w_t$. In other words, $x_1...x_{d_1}$ contains as a subword (as a beginning) one of
 $w_1,...,w_t$, but $x_1...x_{d}$ for any $d < d_1$ does not.

  Next, take the first beginning subword $x_2...x_{d_2}$ of $x_2...x_n$, that is zero in $A$, then the first beginning subword $x_{{d_1}+1}...x_{d_3}$ of $x_{{d_1}+1}...x_n$, that is zero in $A$, then the first beginning subword $x_{{d_2}+1}...x_{d_4}$ of $x_{{d_2}+1}...x_n$, that is zero in $A$,
 etc. This produces a sequence of numbers $d_1\leq d_2 \leq...\leq d_p$, called a {\it generalised Dyck path}.

 The notion of a conventional Dyck path can be obtained from the above, if we start searches of zero subwords not from positions $1,2,d_1 + 1, d_2 +1,...$, but from all subsequent positions $1,2,3,4...$ in the word.  In this latter version the Dyck path is a well-known  and remarkable combinatorial object \cite{C}. The number of Dyck paths of order $p$ is a Catalan number:
 $$C_n=\frac{1}{n+1} \binom{2n}{n}$$

 As we only  deal  with {\it generalised Dyck paths} here  we will sometimes call them just {\it Dyck path}, slightly abusing the terminology.

  We say, that {\it the Dyck path} $d_1\leq d_2 \leq...\leq d_p$ has   {\it length} $r$
  if there are $r$ different numbers in this sequence.

 \begin{theorem}\label{th2.1}
 Let $B_w$:
$$ 0 \mathop{\longrightarrow}\limits^{D_{k+1}} B_k \mathop{\longrightarrow}\limits^{D_k} B_{k-1} \mathop{\longrightarrow}\limits^{D_{k-1}} ... \mathop{\longrightarrow}\limits^{D_2} B_1 \mathop{\longrightarrow}\limits^{D_1} 0$$

be a non-exact complex associated with the word $w =x_1...x_n$ in a monomial algebra $A$, and $r$ be the length of the corresponding generalised Dyck path. Then the complex $B$ has its nonzero homology on the place $n-r$.

 More precisely, if we denote

$H_1= {\ker} D_k / {\im} D_{k+1}$, $H_2={\ker} D_{k-1} / {\im} D_{k}$,..., $H_i = {\ker} D_{k-i+1}   / {\im} D_{k-i+2}, ...,$ $ H_k= {\ker} D_1 / {\im} D_{2},$
then

 $$H_{n-r}=1, H_i=0  \quad (i=1,..,k, i \neq n-r).$$
 \end{theorem}

\begin{lemma}\label{l2.4}
 If there exists a letter $x_i$ in the monomial $w=x_1...x_n$ which is not present in the defining relations $w_1,...,w_t$, then the complex associated to the monomial $w=x_1...x_n$ is exact.
 \end{lemma}

 \begin{proof}
 Note the following: without loss of generality we can suppose that all letters in $w=x_1...x_n$ are different.
 Indeed, the defined above complex
 only depends
 of the positions where the words $w_1,...,w_t$ appear in the word  $w$.
 Therefore, for an arbitrary word $w$ and its subwords  $w_1,...,w_t$, one can take a different word $w'$, which has  the property that all of its letters are different, and  take  subwords
 $w'_1,...,w'_t$ of  $w'$ as a relations that sit in the same positions as    $w_1,...,w_t$ in $w$. Then the corresponding complexes would coincide:

 $$B_{w,w_1,...,w_t}=B_{w',w_1',...,w_t'}.$$

 Thus, the study of a subclass of all monomial algebras, consisting of those algebras, where each generator appears in the relations at most once, is sufficient to study  the properties of the bar complex.

 Take a letter $x_i \in w$ which is not present in any relation   $w_1,...,w_t.$ Suppose it is not the last letter in $w: x_i \neq x_n$. Otherwise we can suppose it is not the first letter $x_i \neq x_1$, as it cannot appear twice in the word $w$.

 The complex $B$ is a span of tensors $B=span \{u_1 \o...\o u_N\}$, where $u_j \neq 0$ in $A$.
 So, obviously, if $x_i$ is not the last letter, the space $B$ splits into two parts: the span of those words where $x_i$ is directly followed by a tensor symbol $\o$: ($\,\,\,x_i \o...$), or the span of those, where it is not:  ($\,\,\,x_i ... \o...$). Denote them as $V_1$ and $V_2$ respectively, so $B=span \{u_1 \o...\o u_N\}=V_1 \o V_2.$

If we define a linear map $\phi:V_1\to V_2$ by $\phi(\dots x_i\otimes u\otimes\dots)=\dots x_iu\otimes\dots$, we see that $\ker \phi=0$ since $x_i$ is not contained in the relations, and that is clearly an onto map, so it is a bijection. This means that $\dim B=2\dim V_1$ is even, hence the complex is exact (corollary from the theorem\ref{th1.1}).
\end{proof}

\begin{proof}[Proof of Theorem\ref{th2.1}] Let us define a sequence of complexes
$$
B,\ L^{(1)},\ L^{(2)},\dots,L^{(r-1)},\ L^{(r)}
$$
inductively, using the definition of complex $L$ from Section\ref{sec2}. Namely, let us split $B$ as $B=E\oplus F$, where
$$
E=\{x_1\otimes\dots\}_k, \qquad F=\{x_1\dots\}_k,
$$
and consider the linear map $J:E\to F$ defined by $x_1\otimes w\mapsto x_1w$. Then
$$
\ker J=\{x_1\otimes\underbrace{x_2\dots x_s}_u\dots\}_k,
$$
where $u=x_2\dots x_s$ is the first (minimal) zero subword of $x_1\dots x_n$ starting with $x_1$. In other words, $s$ is defined as a minimal number, such that $x_1\dots x_s=0$ but $x_1\dots x_{s-1}\neq 0$.

{\bf Definition.} We denote by $L(B)$ the linear space spanned by the monomials defining the $\ker J$:
$$
L(B)=\{\underbrace{x_2\dots x_s}_u\dots\}_k,
$$
such that $\ker J=x_1\otimes L$. As a complex, $L=B_{ux_{s+1}\dots x_n}$ is a bar complex defined by the word $ux_{s+1}\dots x_n$.

We then  set $L^{(1)}=L(B)$ and continue constructing $L^{(2)}=L(L^{(1)})$ and so on. Each of the complexes $L^{(i+1)}$ is obtained as a quotient of $L^{(i)}$

\begin{lemma}\label{com}
The following diagram of complexes is commutative:

$$
\begin{array}{ccccccccccc}
0&\longrightarrow&B_k&\mathop{\longrightarrow}\limits^{D}&B_{k-1}&\longrightarrow&
\dots&\longrightarrow&B_1&\longrightarrow&0\\
&&\downarrow&&\downarrow&&&&\downarrow&&\\
0&\longrightarrow&L_k^{(1)}&\mathop{\longrightarrow}\limits^{d^{(1)}}&L_{k-1}^{(1)}
&\longrightarrow&\dots&\longrightarrow&L_1^{(1)}&\longrightarrow&0\\
&&\downarrow&&\downarrow&&&&\downarrow&&\\
&&\vdots&&\vdots&&&&\vdots&&\\
&&\downarrow&&\downarrow&&&&\downarrow&&\\
0&\longrightarrow&L_k^{(r-1)}&\mathop{\longrightarrow}\limits^{d^{(r-1)}}&L_{k-1}^{(r-1)}&
\longrightarrow&\dots&\longrightarrow&L_1^{(r-1)}&\longrightarrow&0\\
&&\downarrow&&\downarrow&&&&\downarrow&&\\
0&\longrightarrow&L_k^{(r)}&\mathop{\longrightarrow}\limits^{d^{(r)}}&L_{k-1}^{(r)}&
\longrightarrow&\dots&\longrightarrow&L_1^{(r)}&\longrightarrow&0
\end{array}
$$
\end{lemma}

\begin{proof}

Consider the square of the diagram:

$$
\begin{array}{rcccl}
&B_k&\mathop{\longrightarrow}\limits^{D_k}&B_{k-1}&\\
\llap{\hbox{$M_k$}}&\downarrow&&\downarrow&M_{k-1}\\
&L_k&\mathop{\longrightarrow}\limits^{d^{(1)}_k}&L_{k-1}
\end{array}
$$

Here, maps $M_k: B_k \longrightarrow L_k$ are defined by $x_1 \o x_2...x_{s_1}... \mapsto u_1...$, where $u_1=x_2...x_{s_1}$, while all other monomials are mapped to zero.

Then
$$ D_k(x_1 \o x_2...x_{s_1}w_1 \o w_2 \o w_3...)=$$
 $$x_1 x_2...x_{s_1}w_1 \o w_2... - x_1 \o x_2...x_{s_1}w_1  w_2 \o w_3... + ...$$

 Applying $M$ to this polynomial, we get
 $$M_k(D_k))=-u_1 w_1 w_2 \o w_3...+ u_1 w_1 \o w_2 w_3... - ...$$

 Now we do it the other way around:
 $$M_k(x_1 \o x_2...x_{s_1}w_1 \o w_2 \o w_3...)=u_1w_1 \o w_2 \o w_3...$$
 $$d_k(M_k(x_1 \o x_2...x_{s_1}w_1 \o w_2 \o w_3...))=d_k(u_1w_1 \o w_2 \o w_3...)=$$ $$u_1w_1 w_2 \o w_3... - u_1w_1 \o w_2  w_3... + ...$$

 So, we see that $M(D(v))=-d(M(v))$ for any   'nontrivial'  monomial
 $v=x_1 \o x_2...x_{s_1}...$.

 Same type of argument works for any other row of the diagram as long as  formulas for the maps are substituted accordingly.
\end{proof}

The $n^{\rm th}$ row of the diagram is defined by the $n^{\rm th}$ element of the Dyck path, counted without repetitions.
Note that a number of spaces at the beginning of $L^{(i)}$ are zero. It is our first goal is to calculate where the first non-zero space in each row appears. For $L^{(1)}=B_{u_1x_{s_1+1}\dots x_n}$, where $u_1=x_2\dots x_{s_1}$ is defined by $x_1\dots x_{s_1}$ being the first  zero subword $x_1\dots x_{s_1}$ of $\omega=x_1\dots x_n$ that starts with $x_1$. To get elements of this kind inside $B$, taking into account that each application of the differential cuts out one $\otimes$ symbol, we need to take $s_1-2$ steps. So, $L_i^{(1)}$ is non-zero after $s_1-2$ steps from the left of $L_k^{(1)}$. Since $L^{(2)}=B_{u_2x_{s_2+1}\dots x_n}$ where $u_2=x_3\dots x_{s_1+1}\dots x_{s_2}$, is defined by having  $x_2\dots x_{s_2}$ as the first  zero subword of $x_2\dots x_n$ that starts with $x_2$, to get elements of this kind into $B$ we need to cut out $s_2-s_1-1$ extra tensor symbols. Therefore, $L_i^{(2)}$ will be non-zero after $s_1-2+s_2-s_1-1=s_2-3$ steps. The number of steps which will be added at the third row is $s_3-s_2-1$, so $L_i^{(3)}$ is non-zero after $s_3-4$ steps and so on.

Now consider the step $L^{(r-1)}$: $L^{(r-1)}=B_{u_{r-1}x_{s_{r-1}+1}\dots x_n}$, where
$u_{r-1}=x_r\dots x_{s_{r-1}}$ is defined as follows: $x_{r}\dots x_{s_{r-1}}$ is the first
 zero subword $x_{r-1}\dots x_{s_{r-1}}$ of $x_{r-1}\dots x_n$ that starts with $x_{r-1}$). After $s_{r-1}-r$ steps $L^{(r-1)}$ is non-zero  and it continues as the bar complex of the word $u_{r-1}x_{s_{r-1}+1}\dots x_n$ of length $n-(s_{r-1}+1)+2=n-s_{r-1}+1$. Since this is the last step in the Dyck path, the only relation we have in the word is the whole word, so the complex is very nearly the "free" complex, with the exception that at the last term we have a zero space instead of the one-dimensional space $\{u_{r-1}x_{s_{r-1}+1}\dots x_n\}_k$. Exactly at this last term, the homology is $1$. To get to the place where the nonzero homology appears, we need to do $M-1$ more steps, where $M$ is the length of the free complex on a word of length $n-s_{r-1}+1$. It is easy to see that $M$ for a word of length $k$ is $k$. Hence we have a non-zero homology at the place $n-s_{r-1}+1-1+s_{r-1}-r=n-r$ from the  left-hand side term $B_k$, found at the beginning of the complex $B$.

In this construction we used that $s_1<\dots<s_r$, and $s_r=n$, otherwise the complex would be exact. Indeed, the fact that the two neighboring numbers  coincide: $s_{i-1}=s_{i}$  means that $x_i$ does not appear in relations, hence
 according to Lemma \ref{l2.4} the complex is exact, a case which is excluded from the statement of the theorem. Analogously, if $s_r \neq n$, then $x_n$ does not appear in relations.
\end{proof}

\begin{corollary}\label{E}
The value of the Euler characteristic of the complex $B_w(A)$, $w \in \lll X \rrr$,

$${\cal E}_{B_w} = \sum (-1)^i H_i(B_w)$$

can be only $0,$ $1$ or $-1$.
\end{corollary}

{\bf Remark}

After we found the proof presented above, we realized that the information needed to answer the question about the value and the place of the homology is contained in the Anick construction of n-chains \cite{An,An1, Ufn}.  Our solution, however, only uses  part of the information from the  n-chains; namely, only the places of the ends of chain's elements  are used and the answer on the place of nonzero homology is given in terms of genearalised Dyck path only.  For example, there could be two different chains with the same place of homology.
Consider, for example, the word $xyzz$ in the monomial algebra given by relations $xyz=zz=0$ and
the word  $xxxx$ in the monomial algebra given by relation $xxx=0$. The 2-chains in these two situations are different: $1-3, 3-4$ in the first case, and $1-3, 2-4$ in the second case, but the sequence of ends of chain elements are the same: $d_1=3, d_2=4$.
Consequently, our argument itself is more straightforward, because it does not deal with an extra information, as in the case of n-chains, where arguments are different and much more involved.

\section{Inverting noncommutative series associated to monomial algebras}\label{inv}

As a consequence of Theorem\ref{th1.1}, and actually already of a corollary\ref{E}, the following funny fact can be obtained as an application.

\begin{corollary}\label{m}
Let $A$ be a monomial algebra as above. It has a natural grading by the free noncommutative monoid $\lll X \rrr = \lll x_1,...,x_n \rrr,\,$
$A=\mathop{\oplus}\limits_{m \in \lll X \rrr} A_m$. Then the series of $A$ associated to this grading
$$R=\sum\limits_{m \in \lll X \rrr} \rm{dim} A_m \,\,\, m$$

has a property that all coefficients of $\frac{1}{R}$ are $0,$ $1$ or $-1$.

\end{corollary}

Note that there are only a finite number of series on commuting variables with the same property.

\begin{proof}
The series $R$ can be represented as $R=1+G$, where the series of $G$ corresponds to the augmentation ideal of $A$.

The bar complex of $A$ also is graded by $\lll X \rrr$, and the Poincare series of the bar complex is equal to $-G+G^2-G^3+...$, on the other hand this is exactly an expression for $\frac{1}{R}-1$. Since the Poincare series of the bar complex coincides with the Poincare series of its homology, and  since the coefficients of the latter are shown
(Theorem\ref{th1.1}) to be $0$ and $1$, or $-1$, if taken with negative sign into the series, it follows, that $\frac{1}{R}$ also  has  only coefficients $0,$ $1$ or $-1$.

\end{proof}

\section{Operadic version of dichotomy}\label{sec5}

We consider here bar complex $B_w$ associated to a tree $w$ and its subtrees $w_1,...,w_n$.
This is a subcomplex of the bar complex of (symmetric or nonsymmetric) operad  (see \cite{LV}) presented by tree-monomials $w_1,...,w_n$.
We can keep in mind that the bar complex associated with the tree is a special case (see, for example, \cite{KG}) of the Kontsevich graph complex \cite{KGr}.

The maps in the bar complex can be visualized as follows. Let us depict trees from the free operad as trees with (marked) tensors in the vertices. The maps in the bar complex send a tree to the linear combination of trees, where one tensor is substituted by bullet, with signs assigned according to the  Koszul rule. This substitution in our interpretation will serve as an analogue for performing operations or edge contractions.
Whenever the tree contains a subtree $w_i$  with bullets in all internal  vertices, it becomes zero in the monomial operad.

First of all, we consider the question of dichotomy in homology
 in the operadic setting.

\subsection{The 0-1 dichotomy in the homology of monomial operads fails
in the operadic setting }

Here we construct an example of a monomial operad with a pure homology that is concentrated in one place but that is not equal to 1.


{\bf Example}\label{tree3}

Consider the complex $B_w$ for the following tree $w$:
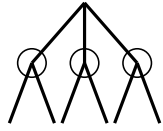
$$
\begin{pspicture}(0,0)(5,3)
\psline[linewidth=1.2pt](2.5,3)(2.5,2.2)
\psline[linewidth=1.2pt](2.5,3)(1.8,2.2)
\psline[linewidth=1.2pt](2.5,3)(3.2,2.2)
\pscircle[linewidth=.6pt](2.5,2.2){0.2}
\pscircle[linewidth=.6pt](1.8,2.2){0.2}
\pscircle[linewidth=.6pt](3.2,2.2){0.2}
\psline[linewidth=1.2pt](2.5,2.2)(2.2,1.4)
\psline[linewidth=1.2pt](2.5,2.2)(2.8,1.4)
\psline[linewidth=1.2pt](1.8,2.2)(1.5,1.4)
\psline[linewidth=1.2pt](1.8,2.2)(2.1,1.4)
\psline[linewidth=1.2pt](3.2,2.2)(2.9,1.4)
\psline[linewidth=1.2pt](3.2,2.2)(3.5,1.4)
\end{pspicture}
$$
The monomial relations are $w_1, w_2,w_3:$

$$
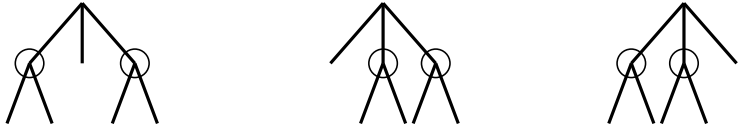
\begin{pspicture}(0,0)(12,3)
\psline[linewidth=1.2pt](2.5,3)(2.5,2.2)
\psline[linewidth=1.2pt](2.5,3)(1.8,2.2)
\psline[linewidth=1.2pt](2.5,3)(3.2,2.2)
\pscircle[linewidth=.6pt](1.8,2.2){0.2}
\pscircle[linewidth=.6pt](3.2,2.2){0.2}
\psline[linewidth=1.2pt](1.8,2.2)(1.5,1.4)
\psline[linewidth=1.2pt](1.8,2.2)(2.1,1.4)
\psline[linewidth=1.2pt](3.2,2.2)(2.9,1.4)
\psline[linewidth=1.2pt](3.2,2.2)(3.5,1.4)
\psline[linewidth=1.2pt](6.5,3)(6.5,2.2)
\psline[linewidth=1.2pt](6.5,3)(5.8,2.2)
\psline[linewidth=1.2pt](6.5,3)(7.2,2.2)
\pscircle[linewidth=.6pt](6.5,2.2){0.2}
\pscircle[linewidth=.6pt](7.2,2.2){0.2}
\psline[linewidth=1.2pt](6.5,2.2)(6.2,1.4)
\psline[linewidth=1.2pt](6.5,2.2)(6.8,1.4)
\psline[linewidth=1.2pt](7.2,2.2)(6.9,1.4)
\psline[linewidth=1.2pt](7.2,2.2)(7.5,1.4)
\psline[linewidth=1.2pt](10.5,3)(10.5,2.2)
\psline[linewidth=1.2pt](10.5,3)(9.8,2.2)
\psline[linewidth=1.2pt](10.5,3)(11.2,2.2)
\pscircle[linewidth=.6pt](10.5,2.2){0.2}
\pscircle[linewidth=.6pt](9.8,2.2){0.2}
\psline[linewidth=1.2pt](10.5,2.2)(10.2,1.4)
\psline[linewidth=1.2pt](10.5,2.2)(10.8,1.4)
\psline[linewidth=1.2pt](9.8,2.2)(9.5,1.4)
\psline[linewidth=1.2pt](9.8,2.2)(10.1,1.4)
\end{pspicture}
$$
In the complex $B_w$, we have $\dim B_3 = 1$, $\dim B_2 = 3$, and $\dim B_1 = 0$, hence we know that $\dim H_2 = 2.$

Later, in section~\ref{sec7}we will see
that homological purity of monomial operads does not hold in general. In later sections we also develop combinatorial tool of transformations of relations graph, which provides a much easier way to calculate the homologies of quadratic operads.





\section{Homological dichotomy  holds for monomial operad given by an order  }\label{sec6}

As we can see from the previous section we must impose some conditions on the presentation of monomial operads in order to get an analogue of the 0-1 dichotomy result for monomial algebras.

\begin{definition} We call the tree $w$ with the set of subtrees $w_i$  {\it an order} if the following property is satisfied.
There is a corolla at the bottom of the tree, which is contained in a unique minimal (by inclusion) tree $w_k$ from the set of subtrees $w_i$. Moreover, the same property should be satisfied for the tree $\bar w$ and the family of subtrees $w_i$ obtained from $(w,w_1,...,w_n)$ by substituting the tree $w_k$ with the corolla (performing all its operations). Note that corresponding substitutions are also performed  in all of the remaining trees that intersect with the tree $w_k$ thus transforming them into set $\bar w_i$.
\end{definition}

\begin{theorem}\label{th1.1}
 The full homology of the above defined  subcomplex $B_w$ of the bar complex in the monomial operad associated to an order $(w,w_1,...,w_n)$  is at most one-dimensional:
$$ \dim H_{\bullet}(B_w)={\dim} {Ker} D / {Im} D \in \{0,1\}$$
 \end{theorem}

 The proof  reconstructs the argument from section~1 in the operadic setting exactly.

 \begin{proof}
The prove is again by induction.
 Basis: if the tree word $w$ is empty, the complex is zero and $H(B_w)=0$. If $w=x$ is a corolla, the complex is $0 \longrightarrow k \longrightarrow 0$, and $H(B_w))=k$ is one-dimensional.

 We need to prove:
 ${\dim} {Ker} D / {Im} D \in \{0,1\}.$

 Since
 $ {\dim}  {Im} D   = {\dim} B -     {\dim} {Ker} D,
 $
 it is the same as
$$ {\dim} {Ker} D -  {\dim} {Im} D = $$
 $$ 2 {\dim} {Ker} D -  {\dim} B \in \{0,1\}
 $$

Thus, $\dim H_{\bullet}(B_w,D)=2 {\dim} {Ker} D -  {\dim} B $.

 Let us start from the bottom of our tree, and call the corolla at the bottom of the tree that belongs to the unique subtree from the set of relations $w_i$, $x_1$. Then, we can   split the vector space $B$ as $B=E \oplus F$, where $E=\{x_1\o...\}_k$ and
 $F=\{x_1\star...\}_k$ are subspaces spanned by those tree-monomials where the first corolla is followed directly by a tensor symbol, and by those where it is not respectively.
  Here $\star$ means that in this place and further up the tree operations are either $\o$ or bullet $\cdot$, which stays for the multiplication.

 Consider the linear map
 $ J: E \longrightarrow F$, defined by turning the tensor after this first corolla into a bullet:  $x_1\o u \mapsto x_1u$.

 The kernel of this map is a linear span of tree-monomials of the type $x_1\o x_2...x_s\star ...$,
 where $x_1x_2...x_s $ is this minimal (by inclusion) unique tree-monomial containing corolla $x_1$, which  exists according to the definition of an order.
$${Ker} J=\{x_1\o x_2...x_s\star... \}_k$$

 We denote by $L=\{x_2...x_s\star... \}_k$, where $x_2...x_s$ as above, so ${Ker} J=x_1 \o L$. Essentially, we find the space $L$ that describes our ${Ker}$, which is the key point of our argument.

 First, we calculate the images of differential $D$ on elements from $E$ and $F$, namely:
$$D(x_1\o u)= x_1u- x_1 \o du= J(x_1 \o u) - x_1\o du$$

 Here $d$ is  the same differential as $D$, just acting on the tree-monomial $x_2...x_n$ instead:
 $$D(x_1 u)= x_1 du =J (x_1 \o du).$$

 Calculate now the ${Ker} D$. Let $\ep \in B$ such that $D(\ep)= 0$.
 We know that $B=E \oplus F$, so $\ep $ is uniquely presented as:
 $$ \ep = x_1 u + x_1 \o v.$$

 Applying $D$ to elements from $E$ and $F$ as above, we have
$$0= D(\ep)= J(x_1\o du)+J(x_1 \o v) - x_1 \o dv.$$

 Note that here $J(x_1\o du)+J(x_1 \o v) \in F$ and $x_1 \o dv \in E$.
 Hence this is equivalent to $x_1 \o dv = 0 $ and $J(x_1 \o (du + v))=0$.
 Which means $dv=0$ and $ x_1 \o (du + v) \in {Ker} J=x_1 \o L$.
 Again, the latter two conditions are equivalent to $dv=0$ and $du+v \in L$, which in turn can be reformulated as $dv=0$ and $v\in -du+ L$. The latter two mean $v\in -du+ (L \cap {Ker} d)$.

 Now,
 $$ {\dim} {Ker} D = \dim F + \dim(L \cap Ker d)= \dim E-\dim L+\dim(L \cap Ker d).$$

 Thus
 $$ {\dim} H_{\bullet}(B_w,D)=2\dim Ker D -\dim B= $$
 $$2 \dim  E- 2\dim L + 2\dim (L \cap Ker d)- \dim E- \dim E +\dim L= $$
 $$2\dim (L \cap Ker d) - \dim L = {\dim} H_{\bullet}(L_{w'},d).$$

 Here $d$ is a  bar differential on the tree submonomial  of $w$; starting from corolla $x_2$, $w'=ux_{s+1}...x_n$ and $L_{w'}$ is a bar subcomplex defined by the word $w'$.

 We got that $H_{\bullet}(B_w,D) =  H_{\bullet}(L_{w'},d)$, the latter by inductive assumption is in $\{0,1\}$, thus so is $H_{\bullet}(B_w,D)$.
\end{proof}

In the next section, to clarify the picture and to give a more general tool for the calculation of operadic homology (in the case of monomial operads), we first explain  correspondence between the latter and the homologies of certain complexes for monomial factors of Grassmann algebras.

Then in section \ref{sec9} we consider the case of quadratic operads. For them, we define a graph encoding their relations. In terms of these graphs we describe certain transformations that either  preserve the homology or split it as a sum of homologies of smaller complexes. Using this tool we also
 recursively calculate homologies for the operad of truncated binary trees.

\section{Monomial quotients of Grassmann algebra}\label{sec7}

In this section we explain how homologies of the complex $B_w$ associated with a tree-monomial $w$ in  a monomial operad is in a bijective correspondence with homologies of a quotient of the Grassmann algebra by certain set of monomials.

The latter can be completely characterised in some cases in combinatorial terms, as we will do in the following sections. By this means a number of previously obtained results  for homologies of operads \cite{D} can be recovered. New conditions on the dichotomy in homology and homological purity for operads are obtained as well.

The following statement establish a bijective correspondence  between operadic and Grassmann homologies.

 For any operad ${\cal P}=({\cal V}, {\cal R})$ presented by the set of generators  ${\cal V}=\{v_1,...,v_n\}$ and relations $ {\cal R} = \{w_1,...,w_r\}$, which are tree-monomials, the complex ${\cal B}_{w,w_1,...,w_r}$  defined for a tree-monomial $w$ is quasi-isomorphic to the following complex
${\cal B}_{w,w_1,...,w_r}^{Gr}$ in the quotient of the Grassmann algebra. Let $A$ be a Grassmann algebra $A=\lll x_v, v \in {\cal V} \vert x_{v_i} x_{v_j} = - x_{v_j} x_{v_i}, \,\, \forall i \geq j \rrr$ and $\bar A=A/M$ be the quotient by the ideal $M$ generated by monomials $w_1,...,w_r$, considered as elements of $A$ (that is defined  just by its set of variables).

\begin{proposition}
For any fixed monomial $w$ we define a complex of Grassmann algebras
${\cal B}_{w,w_1,...,w_r}^{Gr}$ with the following differential:
$d(u)=u(x_1+...+x_n), \,\, \forall u \in \bar A $.
Then,
$H_{\bullet }({\cal B}_{w,w_1,...,w_r})=H_{\bullet}({\cal B}_{w,w_1,...,w_r}^{Gr}(\bar A)).$

Moreover the homology of any monomial quotient of Grassmann algebra can be presented as a homology of some monomial operad.

\end{proposition}

\begin{proof} The first part of the proposition is obvious and seems to be well-known.

The second part we can prove by presenting an operad with the same homology as any given monomial quotient of the  Grassmann algebra. The data which defines the homology of a monomial quotient of a Grassmann algebra $\bar A =A/M$ is the set of generators of $A$: $v_1,...,v_n$, and each generating monomial  $w_i \in M$ is defined by the subset $\{ v_{i_1},...,v_{i_k} \} \subset \{ v_1,...,v_n \}$. Thus the set $\{ v_1,...,v_n \}$ and $r$ of its arbitrary subsets define the homology of a Grassmann algebra.

Using this data we construct the following operad  to have the same homology as of
${\cal B}_{w}$. Take an operad with one $n$-ary operation  and $n$ binary operations,
a tree $w$ of the shape:

$$
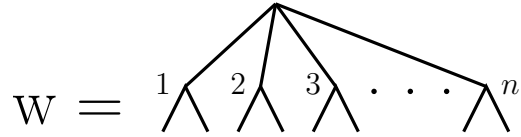
\PICTEN
$$

and relations $w_1,...,w_r$ will be defined by subsets $w_i$ corresponding to subsets
$\{v_{i_1},...,v_{i_k} \} \subset \{v_1,...,v_n \}$

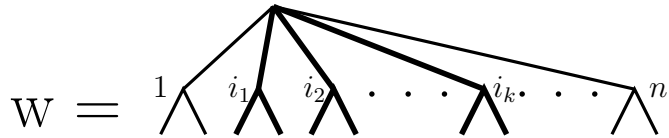
$$
\PICELEVEN
$$


\end{proof}

This means that in terms of homologies of ${\cal B}_w$ in monomial operad, these trees describe everything.

Now we present examples, showing that an operadic homology does not satisfy 1-0 dichotomy, nor homological purity. In this case we calculate homologies using the language of Grassmann algebras.

\begin{example}\label{71}
The following operadic complex ${\cal B}_{w}$
gives an example of non-pure homology: $H_1({\cal B}_{w})=H_2({\cal B}_{w})=1\neq 0$

$
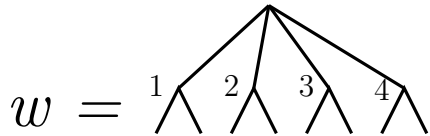
\PICTWELVE
$

{\rm
The relations are: $(123)(14)(24)(34)$

The complex  ${\cal B}_{w}$ is}

$0 \longrightarrow V_1 \longrightarrow V_2 \longrightarrow 0$

$H_1=H_2=1\neq 0$

\end{example}

\begin{example}\label{72}
Let us consider complex ${\cal B}_{w}(w_1,...,w_6) $  defined for the full binary tree and set of monomial relation, consisting of the trees
which contain three binary operations:

$$
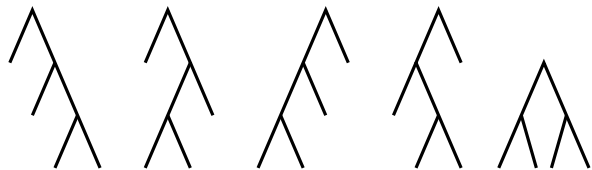
\PICTHIRTEEN
$$


Then $ \dim H_1=0, \dim H_2=3$.
\end{example}

In this case the homology is pure, but not equal to 0 or 1.

The  corresponding Grassmann algebra for the tree-monomial $w$

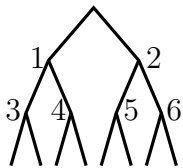
$$
\PICFOURTEEN
$$


is the following:

$\bar A_{Gr}= \lll x_1x_3=x_1x_4=x_1x_2=x_2x_5=x_2x_6=x_3x_4=x_5x_6=0 \rrr.$

The differential $d\xi = \xi (x_1+...+x_6)$ on $x_i$ acts as follows:

$x_1 \to x_1x_5+x_1x_6$

$x_2 \to x_2x_3+x_2x_4$

$x_3 \to -x_2x_3 + x_3x_5 + x_3x_6$

$x_3 \to -x_2x_3 + x_3x_5 + x_3x_6$

$x_4 \to -x_2x_4 + x_4x_5 + x_4x_6$

$x_5 \to -x_1x_5 - x_3x_5 - x_4x_5$

$x_6 \to -x_1x_6 - x_3x_6 - x_4x_6$

$V_2 = \lll x_1x_5, x_1x_6, x_2x_3, x_2x_4, x_3x_5, x_3x_6, x_4x_5, x_4x_6 \rrr_k,$

amongst $d(x_1),...,d(x_6)$ there are 5 linearly independent, so
$ \dim H_1=0, \dim H_2=3$.

We now have an example where homological purity holds, but not the 0-1 dichotomy for the dimension of the homology.

Note that the typical linear-algebraic method for calculating  homologies that we used here is difficult to extend to trees where data is more complicated. Motivated by this problem, we developed in section~\ref{sec9} some quite different methods. These methods are based in arguments of the same nature as in theorem~\ref{th1.1}, consisting of splitting the homology in accordance with certain transformations of graph relations.

\section{Combinatorial conditions for dichotomy in terms of Grassmann algebras}\label{sec8}

We aim to formulate here the condition of being an {\it order} from section \ref{sec6} in combinatorial terms, as well as some other conditions to ensure that the dichotomy in homology will hold.

First, let us define the following recurrently formulated condition on the set $X=\{x_1,...,x_n\}$ with its subsets $R_1,...,R_n$.

We suppose throughout this section that subsets $R_1,...,R_n$ are reduced, that is, that  there are no $R_i \subset R_j$ for $i\neq j$.
This can be obtained by throwing away unnecessary relations that follow from others.

\begin{definition}\label{def1} We call the system $X, R_1,...,R_m$, $X=\{x_1,...,x_n\}$,  $R_i \subset X$ {\it basic}, if
there is a point $x_1 \in X$ which belongs to exactly one subset from $R_1,...,R_m$,
IN this case, we call the system {\it $x_1$-basic}.
\end{definition}

\begin{definition}\label{def2} We say that the system $(X, R_1,...,R_m)$ is {\it an order} if:

(1) $(X, R_1,...,R_m)$ is basic

(2) there is a sequence $x_{i_1},...,x_{i_m}$ with $x_{i_k} \in X$, such that consequent
$x_{i_k}$-contraction transformations of
  $(X, R_1,...,R_m)$ consist of basic systems, and the result of the last transformation  is a one-point system: $(X=\{x_i\}, R=\{x_i\})$.
\end{definition}

\begin{definition}\label{def3} Let  $(X, R_1,...,R_m)$ be a basic system, a system $(\bar X, \bar R_1,...,\bar R_t)$ is called its {\it contraction transformation}, if it is obtained in the following way.

We choose a point $x_1 \in X$ which is contained in a unique subset $R_1 \in \{R_1,...,R_m\}$.

Next we contract the set $R_1$ into a point, which means that
the new set of points is $\bar X = \{\bar x_1, x_2,...,\hat x_j,...,x_n\bar\}$, where we throw away all points $x_j \in R_1, j\neq 1$.

The new set of subsets $\bar R_j$ are:

$\bar R_j = R_j$ if $R_J \cap R_1 = \emptyset$

$\bar R_j = (R_j \cap \{\bar x_1\}) \ \{x_s | x_s \in R_1, s \neq 1 \}$,
if $R_j \cap R_1 \neq \emptyset$.

We call this an {\it $x_1$-contraction transformation}.
\end{definition}

This definition of an order is a combinatorial reformulation of the definition of an order from section \ref{sec6}.

\begin{proposition}
If there is a point $x_i \in X$ which does not  belong in any of  $R_j$, then the homology of the Grassmann algebra defined by the data set $(X,R_i)$ is zero.
\end{proposition}

\begin{theorem}
If $(X, R_1,...,R_m)$ is an order,  then the homology of the Grassmann algebra defined by this data satisfies 0-1 dichotomy.
\end{theorem}

Another, more constructive way to define the same condition is the following.

\begin{definition}\label{def4}  The system $(X, R_1,...,R_m)$ is an order if  there is an ordering of subsets $\{R_i\}$
$R_{i_1},...,R_{i_m}$ such that

(1) there is a point $x_1\in X$ which belongs only to  $R_{i_1}$.

(2) there is a point $x_2\in X$ which belongs only to  $R_{i_2}$ or

$ \bar R_{i_2} \cap R_{i_1} \neq \emptyset,  R_{i_k} \cap \bar R_{i_2} = \emptyset, \, \forall k \geq 3.$

(the latter corresponds to a situation where the unique point does not exist in $R_{i_2}$ but will appear  - in $\bar R_{i_2}$ - after the contraction transformation).

(3) there is a point $x_3\in X$, which belongs only to  $R_{i_3}$ or

$ R_{i_3} \cap \bar R_{i_2} \neq \emptyset,  R_{i_k} \cap \bar R_{i_2} = \emptyset, \, \forall k \geq 4,$

where $\bar R_{i_2}$  is a set obtained from $R_{i_2}$ after $x_2$-contraction transformation).

[In terms of original sets it means

 $ R_{i_3} \cap  R_{i_2} \neq \emptyset$ or $( R_{i_1} \cap  R_{i_2} \neq \emptyset$ and
$R_{i_1} \cap  R_{i_3} \neq \emptyset)$

 and

 $ R_{i_k} \cap  R_{i_2} \neq \emptyset$ and $( R_{i_1} \cap  R_{i_2} = \emptyset$ or
$...R_{i_1} \cap  R_{i_k} = \emptyset) \forall k \geq 4$.]

(...)

(n) there is a point $x_n \in X$, which belongs only to  $R_{i_n}$ or

$R_{i_{n-1}} \cap \bar R_{i_2} \neq \emptyset$, where $ \bar R_{i_{n-1}}$
 is a set obtained from $R_{i_{n-1}}$ after $x_{i_{n-1}}$-contraction transformation).

\end{definition}

We list the cases where the above condition of being an order is satisfied.

\begin{corollary} Particular cases of orders for which the dichotomy theorem is true are the following.

\proof

(1) For any set $R_i$ there is a point unique to this set.

 Definition \ref{def4} is obviously satisfied.

(2) there are sets $A_1^*,...,A_k^* \in \{ R_i \}$ which  have a unique point and sets $S_1,...,S_l \in \{ R_i \}$ which do not.

The condition is that there exists $S_0^* \in \{A_1^*,...,A_k^*\} $ such that $S_0 \cap S_1 \neq \emptyset,
S_0 \cap S_i =  \emptyset, \, \forall i \geq 2$, $(S_0 \cup S_1) \cap S_2 \neq \emptyset,  $
$(S_0 \cup S_1) \cap S_i = \emptyset, \, \forall i\geq 3,$ etc.

\proof This condition ensures that if not all sets have a unique points, they will acquire them after a certain number of contraction transformations (in a certain way).

(3) The elements of the set $\{R_i\}$ can be positioned at the vertices of the tree with the following property.

(a) two vertices $i$ and $j$ are joined by joined by an edge if $R_i \cap R_j \neq \emptyset$.

(b) the leaves of the tree are exactly the subsets $R_i$th which do have a unique element.

\proof After contracting  each  leaf-set we  again  get a tree with the same property.
Thus, repeating the process, we eventually can contract the tree to one point $(X=\{x\}, R=\{x\})$, which is a stable configuration under the contraction transformation with the homology 1.

\end{corollary}

The property of being an order, and stronger conditions $(1)-(3)$ are generalisations of the property of being a chain in the case of subsets $R_1,...,R_n$ which are intervals within the word.

\section{Homologies of quadratic operads}\label{sec9}

We describe here some  operations on homology in terms of transformations of graphs which encode relations for a quadratic monomial operad.

Using these techniques, one can calculate operadic homologies in certain examples. We present here an example of such a calculation for the quotient of binary tree by all degree two relations (which involve three binary operations), a so called truncated free binary operad $T^{(3)}=T/V^{(3)}$.

The data is as before: a tree $T$ on a set of vertices $X$, usually finite, with a set of subtrees, uniquely defined by subsets of vertices $X_1,...,X_r$, $X_i \subset X$.

The case we consider in this section is when all subsets $X_i$-th have no more than two elements. It is the case of a monomial quadratic operad, i.e. the operad is  given by tree-relations with two internal vertices.

This data can be encoded by the {\it relations graph} $G_R = G_{T,X,\bar X_i}$ with the set of the internal vertices from $X$, as a set of vertices. There is an edge between vertices $x_i$ and $x_j$ in $G_R$ if and only if there is a subset $X_k \subset X$, $X_k=\{x_i,x_j\}$, that is these two vertices form a relation.

The following facts about the homologies of operads that are defined by generators $X$ and relations $X_1,...,X_r$ hold (in terms of relation graph $G_R$).

\begin{proposition}\label{pr0} If there is an isolated vertex $x \in X$ in $G_R$
then the homology is zero.

\end{proposition}

\begin{proposition}\label{pr1}
If there are connected components $G_1,...,G_n$ of $G:$ $G = \sqcup G_i$, then
$$H^{\bullet} (A_G)=
H^{\bullet} (A_{G_1})...H^{\bullet} (A_{G_N}).$$
\end{proposition}

\proof It is a direct consequence of the K\"unneth theorem on the homology of the tensor product of complexes.

The following theorem provides a tool for homology calculations in some interesting examples.
The described below  transformations of the graph
$G_R$ correspond to the following operations on homologies.

\begin{theorem}\label{9.3}
Suppose there exists a vertex $x$, connected only to vertices $a_1,...,a_n$, which are  pairwise connected to each other (this set of vertices can consist of one element as well). Then
$$H^{\bullet} (A_G)=H^{\bullet} (A_{G_1}) \otimes ... \otimes  H^{\bullet} (A_{G_n}),$$
where $G_i$ is a graph obtained from $G$ by deleting all vertices connected to $a_i$.

Our convention is that when deleting a vertex, we delete all edges coming out of this vertex, but not the vertices on the other side of the edge.

\end{theorem}

\begin{proof} The proof is a calculation of the kernel of the map of multiplication by $x$:
$\phi_x: A_G \to A_G: u \to xu$. This is analogous to the map $J: E \to F$ from ection \ref{sec2}. By arguments similar to those in Theorem~2.1 this kernel is a complex with the same homology as the  initial one. We present those arguments here in terms of a Grassmann algebra.

\begin{lemma}
Let $V=\{ x_1,...,x_N \}$ be the set of generators of the Grassmann algebra $Gr(V)$ and $X=Gr(V)/\lll r_1,...,r_m \rrr$
be a monomial quotient of $Gr(V)$ by quadratic monomials $r_i$.
It is also a complex with respect to the natural grading, with the derivation  $D: X \to X$, $D(\xi)=(x_1+...+x_N) \xi.$

Then the set
$Y_x=\{ m\in \lll V \rrr, m \neq 0$ in $X \, | \, x $
 is not present in $m, xm = 0$ in $X \}$
is  a subcomplex in $X$ which is quasi-isomorphic to $X$:
 $H^{\bullet}(X) \simeq H^{\bullet}(Y)$.

\end{lemma}

\begin{proof}
First, it is easy to see that $Y$ is indeed a subcomplex of $X$ ($D(Y) \subset Y$):
$D(m)=(x+x_2+...+x_N)m=x_2m+...+x_N m$  is the sum of x-free monomials, and since $xx_jm=-x_jxm=0$,
multiplying by $x$ from on the right produces zero. Thus, $D(m)$ is in $Y$.

Now, we want to show that $  H^{\bullet}(X) \simeq  H^{\bullet}(Y)$.
Let $\xi \in X$ be an element of $ker D$, $D:X \to X$, $\xi=\sum c_j m_j = xf+g$ where $f$ and $g$ are free from $x$.
If $D{\xi}=0$, we have
$$(x+x_2+...+x_N)(xf+g)=0,$$
which means that $(x_2+...+x_N)xf+xg=0$, as the x-part of the above expression, and $(x_2+...+x_N)g=0$ as an x-free part of the above expression. From the first equality, we have $g=(x_2+...+x_N)f+h$
 where $h\in Y$, then $(x_2+...+x_N)g= (x_2+...+x_N)^2f+(x_2+...+x_N)h=(x+x_2+...+x_N)h=D(h)$.
 Thus we see that $\xi \in ker D$ if and only if $\xi = D(f)+ h$, where $f\in X, h \in ker D \cap Y$.
 Hence $  H^{\bullet}(X) \simeq ker D_x \cap Y \simeq  H^{\bullet}(Y)$, and taking care of the grading it is also easy to show that these spaces are isomorphic as a graded vector spaces.
 \end{proof}

 Let $x$ now be a vertex in the graph $G_R$, connected to vertices $x_1,...x_n$, which are pairwise connected to each other. The subcomplex $Y_x=\{ m\in \lll V \rrr, m \neq 0$ in $X \, | \, x $
 is not present in $m, \, xm = 0$ in $X \}$ consists of sets of vertices (or monomials) which have a point from $V$ that is connected to $x$, (as otherwise, $xm \neq 0.$) If this point is $x_i$ we denote the corresponding subset of monomials $Y_{x,i} \subset Y_x$, and note that  $Y_{x,i}$ is a subcomplex in $Y_x$. The homology of $Y_x$ is in general a sum of homologies of $Y_{x,i}$. Taking into account that all vertices $x_i$ are pairwise connected to each other, we see that in fact $Y_x=\oplus Y_{x,i}$, where the direct sum is taken over $x_i$th that are connected to $x$. Thus $\dim H^{\bullet}(Y_x)= \sum \dim  H^{\bullet}(Y_{x,i})$.

 Each complex $Y_{x,i}$ is obviously isomorphic to the complex encoded by the graph of relations $G_i$, obtained from $G$ by deleting all vertices connected to $x_i$ (in a sense that deleting a vertex we delete all of edges connected to it, without deleting vertices on another side of these edges). Indeed, we obtain the same differential in both cases as multiplication by $x_i$ is zero for all of the points connected to $x_i$ in $G$.
 This completes the proof of the Theorem~\ref{9.3}

\end{proof}

\begin{example}\label{9.1}
 {\rm Let $T$ be an infinite binary tree and $T_n$ be  its subtree of depth $n$. Let ${\cal A}_n^{(3)}= T^n / {\cal J}_n^{(3)}$, where  ${\cal J}_n^{(3)}$ is obtained from the set of all subtrees of the binary tree, containing three operations:
$$
{\cal I}_n^{(3)}=\left\{{\vrule width0pt height 1.3cm}\smash{\raisebox{-1cm}{\PICONE}}\right\}.
$$

The ideal ${\cal J}_n^{(3)}$ is generated by all subtrees of elements of ${\cal I}_n^{(3)}$ obtained after cutting $T$ at the level of depth $n$.

For example, the ideal ${\cal J}_3^{(3)}$ is generated by all subtrees from the set
${\cal I}_3^{(3)}$ and those shown in the picture.}
$$
T_3=\smash{\raisebox{-1cm}{\PICTWO}}
$$


\begin{proposition}
$H^{\bullet}({\cal A}_n^{(3)})= 1$
\end{proposition}

\end{example}

\proof We will first consider the full binary tree $T$, then we will ensure that the argument works for any binary tree.
Note first that we can consider a reduced set of relations, which generates the same ideal, but does not contain relations, which are subtrees of each other.
For example, out of the three relations shown in the picture above, we only take the last one to the reduced set of relations.

Now let $G$ be the whole graph and $G'$ be obtained from $G$ by deleting one of the binary operations $\Lambda$ (which is a relation), on the lower level of the tree.

Since it is a relation in a reduced set of relations, vertices of $G$ are divided in two subsets: vertices of $G'$ and vertex of $\Lambda$, so that there are no relations containing vertices from both subsets. This means that we have the situation of Prop.~9.2 $A_G=A_{G'} \otimes A_{\Lambda}$ and so we can apply the K\"uneth formula:
$$H(A_G)=H(A_{G'} \otimes A_{\Lambda})=H(A_{G'}) H(A_{\Lambda}).$$

Since $ H(A_{\Lambda})$ is one-dimensional, deleting  operation $\Lambda$ does not change the homology.
Using this procedure, starting with the full binary tree, we arrive at the tree, with only one corolla -  the binary operation $\Lambda$ at the top of the tree, whose homology is one-dimensional.

\begin{example}{\rm  Let ${\cal C}_n = T_n /{\cal I}_n^{(3)}$, where $T_n$ is a full binary tree of depth $n$.}
\end{example}

The homology of ${\cal C}_n$ depends on $n$ and grows quite fast.
 It can be calculated using Theorem~\ref{9.3}.
  The initial full binary tree $T_n$
 produces the graph of relations:

$$
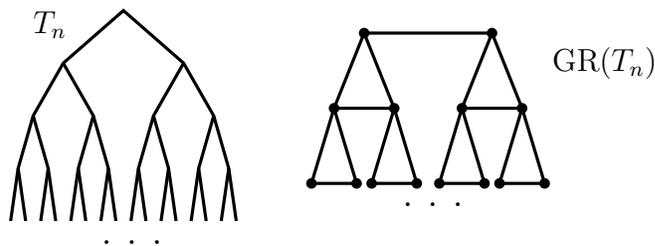
\PICTHREE\,\qquad\qquad
$$


 For example, let us calculate the homology of the full binary tree of depth 2, which has the following graph of relations:

$$
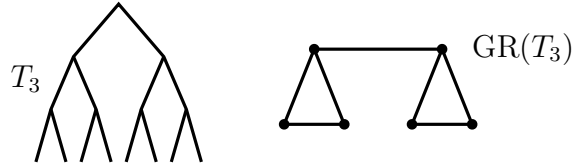
\PICFOUR
$$


 Consider the vertex $a$ which is connected only to two vertices $b$ and $c$, which are connected to each other.
 Then

$$
\PICFIVE
$$

Homology of the triangle is equal to two, and homology of one edge is equal to one, which is easy to see either by direct calculations or considering again one vertex and reducing using Theorem~\ref{9.3} to the homology of empty set, which is equal to one.


\section {On the recurrent formula for $H(T_n)$.}

We are able to calculate $H^{(n)}$ for any $n$ by establishing the  recurrence relation on the array parametrised
by  a pair $n, w(x,y,z)$,
where $n$ is a natural number and $w(x,y,z)$ is a noncommutative monomial on three variables from $\langle  x,y,z \rangle.$

The following graph $G_{n,w(x,y,z)}$ corresponds to a pair $n,w(x,y,z)$:

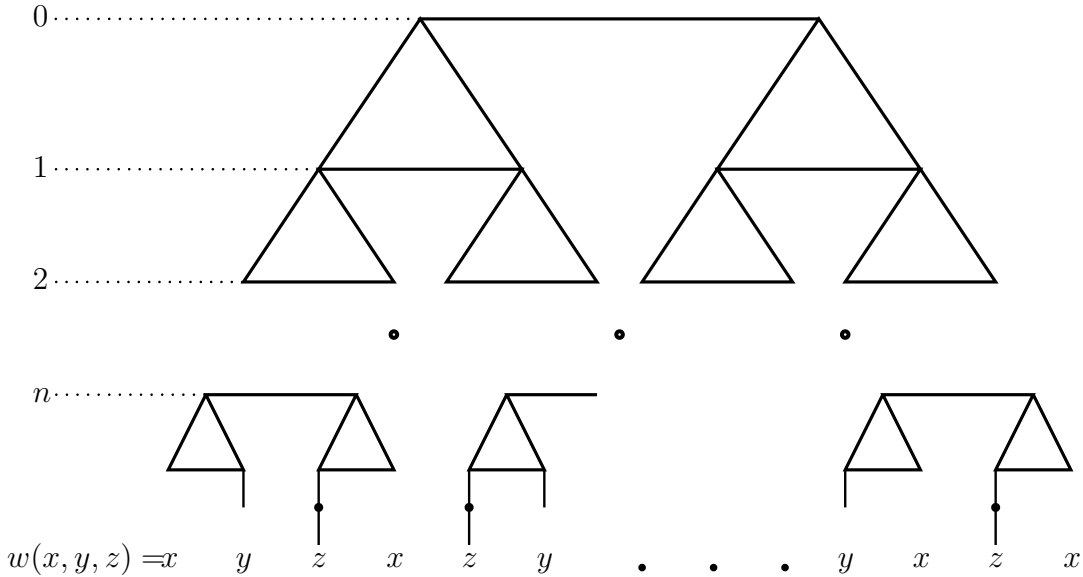
$$
\PICFIFTEEN
$$


The graph for which we would like to know the value of the homology is $G_{n,x^{2^n}}$


Using the main Theorem~\ref{9.3} on graph transformations we can obtain the following rewriting rule, which express homologies of  graphs of depth $n$ via homologies of graphs of depth $n-1$.

\begin{corollary}\label{rew} For any graph $G_{n,w(x,y,z)}$,
$$H(G_{n,w_1\dots w_{2^{n+1}}})$$
$$=H(G_{(n-1), r(w_1w_2)\dots r(w_{2^{n+1}-1}w_{2^{n+1}})}),$$
where $r(x^2)=y+z$, $r(xy)=r(yx)=r(yz)=r(zy)=z$, $r(y^2)=0$ and $r(z^2)=x$.
\end{corollary}

The operation of decreasing depth of the graph of relations will lead to certain expression for the homology:

$$ \dim H(G_{n, x^{2^n}})=a_n \dim H(G_{0,xx}) + b_n  \dim H(G_{0,yy}) + c_n  \dim H(G_{0,zz})$$
 $$+ p_n  \dim (H(G_{0,xy})+H(G_{0,yx})) + q_n  \dim (H(G_{0,xz})+H(G_{0,zx})) +
r_n  \dim (H(G_{0,zy})+H(G_{0,yz})) .$$

We will write simply:
$$ x^{2^n}=a_n x^2+b_n y^2+...+r_n (zy+yz).
$$

Now we will obtain a recurrent formula (of length one) for coefficients
$a_n$, $b_n$, $c_n$, $p_n$, $q_n$ and $r_n$, using that

$$x^{2^{(n+1)}}=a_{n+1} x^2+...=(a_n x^2+...)^2=a_n x^4+...=a_n^2(y+z)^2+...$$
where the latter equation is written taking in account the rewriting rules from the corollary~\ref{rew}.

 The resulting recurrence relations are as follows:
\begin{align*}
&a_{n{+}1}=c_n^2,\ \ b_{n{+}1}=(a_n{+}2q_n)^2,
\\
&c_{n{+}1}=(a_n{+}2p_n{+}2r_n)^2,\ \ p_{n{+}1}=c_n(a_n{+}q_n),\ \
\\
&q_{n{+}1}=c_n(a_n{+}2p_n{+}2r_n),
\\
&r_{n+1}=(a_n+2q_n)(a_n{+}2p_n{+}2r_n)
\end{align*}
with initial conditions
$$
a_1=p_1=q_1=0,\quad b_1=c_1=r_1=1.
$$

Taking into account homologies $H(G{0,xx})$, etc., which are easy to calculate directly, we see that
the homology for an arbitrary $n$ is a sum of solutions of the above recurrence relations with coefficients as follows:
$$
{\rm dim} H^{(n)}= a_n+c_n+2p_n+2r_n.
$$


 \section{Acknowledgements}
  We are grateful to IHES and MPIM for hospitality and support.
 We would like to thank M.Kontsevich for number of enlightening discussions
  during various stages of the work on the paper. Moreover, the argument in the corollary~\ref{m}
is  suggested by M.Kontsevich.
This work was supported  by the European Research Council [grant number 320974]; and the  Engineering and Physical Sciences Research Council [grant EP/M008460/1].

\noindent N.Iyudu

\noindent Max-Planck-Institute f\"ur Mathematik

\noindent Vivatsgasse 7

\noindent 53111 Bonn

\noindent Germany

\noindent  E-mail address: \quad {iyudu@mpim-bonn.mpg.de}\\

\vspace{3mm}

\noindent  Y.Vlassopoulos

\noindent  Institut des Hautes \'Etudes Scientifiques,

\noindent 35 route de Chartres,

\noindent F - 91440 Bures-sur-Yvette

\noindent  E-mail address: \quad {yvlassop@ihes.fr}\\

\vspace{13mm}


\begin{thebibliography}{99 }

\bibitem{An1} D.Anick, \it On the Homology of Associative Algebras,
\rm Transactions of the American Mathematical Society
{\bf 296}, No. 2, 1986 pp. 641-659.

\bibitem{An} D. Anick, \it Generic algebras and CW complexes, \rm Algebraic topology and algebraic K-theory (Princeton, N.J., 1983), 247–321, Ann. of Math. Stud. {\bf 113}, Princeton Univ. Press, Princeton, NJ, 1987.

    \bibitem{L} A.Belov, V.Latyshev, V.Borisenko, {\it Monomial algebras}, VINITI, 1995.

\bibitem{C} R.A.Brualdi, {\it Introductory combinatorics}, New York, Elsevier, 1997.

\bibitem{D} V.Dotsenko, private communication.

\bibitem{Prepr} N.Iyudu, I.Vlassopoulos, {\it Homologies of monomial operads and algebras}, MPIM pteprint 2020-6.


\bibitem{M} M.Jollenbeck, V.Welker, {\it Minimal resolutions via algebraic discrete Morse theory}, Memories AMS, {\bf 932} (2009).


\bibitem{KGr} M.Kontsevich, \it Deformation quantization of Poisson manifolds, I, \rm
 Letters Math.Physics {\bf 66}(2003), no 3, 157--216.

\bibitem{KG} V.Ginzburg, M.Kapranov, \it Koszul duality for operads, \rm Duke Math. J., {\bf 76} (1994), N 1, 203--272.

\bibitem{Ufn} V. Ufnarovskii, \it Combinatorial and asymptotic methods in algebra (Russian), \rm Current problems in mathematics. Fundamental directions 57, 5–177, Itogi Nauki i Tekhniki, Akad. Nauk SSSR, Moscow, 1990.



\bibitem{LV} J-L Loday, B.Vallette, {\it Algebrais operads}, Springer, Grundlehren der mathematischen wissentschaften book series, {\bf 364} (2012).










\vspace{7mm}




\end{thebibliography}
\end{document}